\documentclass[11pt]{article}
\usepackage[textwidth=6.4in,top=1in,bottom=1in,hcentering]{geometry}
\usepackage[utf8]{inputenc}
\usepackage[T1]{fontenc}
\usepackage{lmodern}
\usepackage{amsmath,amsthm,amssymb,mathtools,graphicx,booktabs}
\usepackage{microtype,placeins}
\usepackage{xurl}
\usepackage[colorlinks=true,linkcolor=blue,citecolor=blue,urlcolor=blue]{hyperref}
\hypersetup{pdftitle={Interval Inverse Inequalities for Multiscale Kernel Primitives with Positive Exponential Spectra},pdfsubject={Positive exponential spectra and sharp stability rates}}
\newtheorem{theorem}{Theorem}[section]
\newtheorem{lemma}[theorem]{Lemma}
\newtheorem{proposition}[theorem]{Proposition}
\newtheorem{corollary}[theorem]{Corollary}
\theoremstyle{definition}
\newtheorem{definition}[theorem]{Definition}
\newtheorem{example}[theorem]{Example}

\newcommand{\doi}[1]{\href{https://doi.org/#1}{doi:\nolinkurl{#1}}}
\title{Interval Inverse Inequalities for Multiscale Kernel Primitives with Positive Exponential Spectra}
\author{Chengming Li and Qiming Wang\thanks{Corresponding author. Email: \texttt{wqm@hhu.edu.cn}}\\[3pt]
\small School of Mathematics, Hohai University, Nanjing 210000, Jiangsu, China}
\date{}
\begin{document}
\maketitle

\begin{abstract}
Inverse inequalities are fundamental for controlling derivative growth and ensuring sampling stability in kernel-based approximation. Extending whole-space and single-scale estimates to multiscale analytic spaces on bounded intervals requires uniform control of boundary truncation and cross-scale cancellation. We establish unweighted \(L^2(0,1)\) inverse inequalities of sharp scaling order \(h_*^{-1}\) for trial spaces spanned by affine functions and multiscale primitives of kernels with positive exponential spectra, where \(h_*\) denotes the finest scale. Under geometric scale separation and within-scale center separation, the inverse constant depends only on the kernel and separation parameters; it is independent of the number of scales and centers, allowing centers at the endpoints and arbitrary cross-scale center coincidences. To prove the bound, we combine an exterior Gram decomposition with endpoint-distance weights to establish configuration-uniform Riesz bounds for normalized high-order derivatives of order \(q\), whose lower constant approaches the optimal limit \(1/2\). A synthesis-operator perturbation argument and parameterized interval interpolation then yield the first-order estimate. Two-center constructions confirm sharpness in both the scaling order \(h_*^{-1}\) and the \(q^{-1/2}\) asymptotic deficit from \(1/2\) as \(q \to \infty\). An application to deterministic least-squares sampling is derived, and numerical experiments on boundary-layer asymptotics, multiscale Gram spectra, sharp inverse constants, and midpoint sampling illustrate the theory.

\end{abstract}


\section{Introduction}

Multiscale expansions built from primitives of analytic kernels, such as smooth sigmoidal functions, are widely used to resolve localized features and boundary layers \cite{ref4,ref12}. In these approximations, the active scales often span several orders of magnitude, and the centers naturally cluster near the domain boundary. Establishing inverse inequalities that bound derivatives by the \(L^2\) norm is essential for ensuring numerical stability in Galerkin or collocation discretizations and for guaranteeing stable recovery from discrete samples \cite{ref2}. Obtaining such estimates on a bounded interval \(I = (0,1)\) is nontrivial, however, because restricting a kernel space from the whole line \(\mathbb{R}\) to \(I\) introduces boundary truncation that whole-line estimates alone cannot resolve. Specifically, an \(L^2(\mathbb{R})\)-normalized derivative atom of a real even kernel centered at an endpoint places half of its energy outside \(I\), regardless of the derivative order. Moreover, lower bounds on individual atoms do not preclude destructive cancellation among linear combinations, particularly when local differences of closely spaced primitives form sharp pulses at the finest scale. Both obstructions coexist and interact in multiscale primitive spaces.

On \(I = (0,1)\), the sharp \(L^2(I)\) inverse constant of a finite-dimensional subspace \(V \subset H^1(I)\) is defined by

\[
L(V) = \sup_{0 \ne f \in V} \frac{\left\lVert f' \right\rVert_{L^2(I)}}{\left\lVert f \right\rVert_{L^2(I)}}.
\]

For spaces generated by affine functions and primitives of kernels with positive exponential spectra, we prove \(L(V) \le C h_*^{-1}\) under geometric scale separation and within-scale center separation, where \(h_*\) denotes the finest scale and \(C\) depends only on the fixed kernel and the separation parameters. Centers may lie anywhere in \([0,1]\), with no interior boundary buffer, nesting condition, or cross-scale separation required. Throughout the paper, expansion coefficients are real and all norms are unweighted; we write \(\left\lVert \cdot \right\rVert_2 = \left\lVert \cdot \right\rVert_{L^2(I)}\) unless another domain is specified.

Our proof establishes configuration-uniform Riesz bounds on \(I\) for whole-line normalized high-order derivatives, with a lower constant approaching \(1/2\)—the largest value compatible with endpoint centers. Decomposing the interval Gram matrix as \(G_I = G_{\mathbb{R}} - E_- - E_+\), where \(E_-\) and \(E_+\) denote the exterior half-line Gram matrices on \((-\infty, 0)\) and \((1, \infty)\), respectively, we retain the exterior diagonal energy explicitly and control exterior off-diagonal interactions via one-sided oscillatory and algebraic decay estimates weighted by endpoint distance \(w(u) = (1+u^2)^{-1}\). Because these weights are uniformly summable across admissible centers at each scale, Schur's test yields a quadratic-form lower bound uniform over all coefficient vectors. A whole-line synthesis-operator perturbation argument then extends these high-derivative Riesz bounds from single-exponential spectra to general positive spectral mixtures without requiring a spectral gap above the minimal decay parameter. Combining adjacent-derivative comparison at high orders with a boundary-condition-free Sobolev interpolation inequality on \(I\) yields the first-order inverse inequality for the full primitive space, including its affine component. Figure 1 illustrates the admissible multiscale geometry and interval energy truncation.

Two explicit two-center constructions establish sharpness with respect to scale and derivative order. First, for every fixed kernel in the class, a four-dimensional single-scale space with two interior centers attains the scaling order \(h_*^{-1}\), showing that \(h_*^{-1}\) is optimal and cannot be replaced by the subspace dimension. Second, for a single exponential spectrum \(\vert{}\xi\vert{}^\beta e^{-a\vert{}\xi\vert{}}\), placing one center at an endpoint and a finer-scale center at distance \(O(h/n)\) yields a \(2 \times 2\) interval Gram matrix of \(q\)th-order derivative atoms whose smallest eigenvalue is \(1/2 - \kappa_q n^{-1/2} + O(n^{-1})\) as \(q \to \infty\), where \(n = q + \beta + 1\) and \(\kappa_q > 0\) along both even and odd subsequences of \(q\). Together with our \(O(n^{-1/2})\) uniform upper bound on the eigenvalue deficit, this determines the exact rate at which the best configuration-uniform Riesz lower bound approaches \(1/2\). Finally, the inverse inequality yields quantitative star-discrepancy conditions for deterministic least-squares sampling stability.

\begin{figure}[tbp]
\centering
\includegraphics[width=\linewidth]{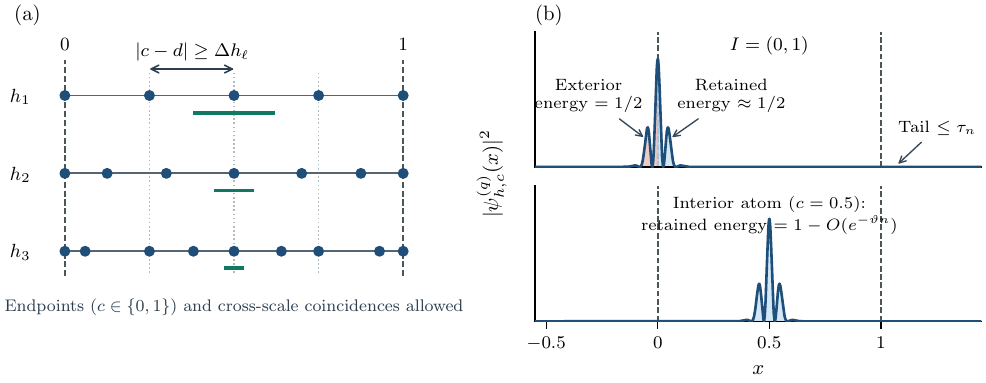}
\caption{\emph{Admissible multiscale configurations and interval energy truncation.} \textbf{(a)} Schematic of multiscale center sets \(X_\ell \subset [0,1]\) across three geometrically separated scales \(h_1, h_2, h_3\) (horizontal bars indicate scale lengths), illustrating endpoint centers (\(c \in \{0,1\}\)), within-scale separation \(\vert{}c - d\vert{} \ge \Delta h_\ell\), and cross-scale center coincidences (vertical dotted lines). \textbf{(b)} Energy densities \(\vert{}\psi_{h,c}^{(q)}(x)\vert{}^2\) of \(L^2(\mathbb{R})\)-normalized single-exponential \(q\)th-derivative atoms (\(a = 1\), \(\beta = 0\), \(q = 8\), \(h = 0.14\), \(n = q + \beta + 1\)) relative to \(I = (0,1)\). Top: an endpoint atom (\(c = 0\)) places half of its energy in \((-\infty, 0)\) and retains \(1/2\) in \(I\) up to an exponentially small far-endpoint tail \(\tau_n\) beyond \(x = 1\). Bottom: an interior atom (\(c = 0.5\)) retains \(1 - O(e^{-\vartheta n})\) of its energy within \(I\) for some \(\vartheta > 0\).}
\label{fig:figure1}
\end{figure}

\subsection{Relation to earlier inverse estimates}

To place our results in context, we compare the present setting with several established directions in the literature. Ward \cite{ref11} proved whole-space \(L^p(\mathbb{R}^d)\) Bernstein inequalities for radial basis functions under algebraic high-frequency Fourier bounds. Mirzaei \cite{ref9} established scale-uniform inverse inequalities for scaled Sobolev kernels, including bounded-domain estimates measured against \(L^2\), when each trial space is spanned by translates at a single scale. Uniformity with respect to a single varying scale differs fundamentally from estimates on spaces combining multiple scales simultaneously, as separate single-scale bounds do not prevent cross-scale cancellation. Moreover, the exponentially decaying spectra considered here violate the algebraic Fourier lower bounds assumed in Sobolev-type settings.

Bounded-domain inverse inequalities have also been obtained for finite-smoothness kernels via localized bases. For Matérn and surface spline kernels, Hangelbroek, Narcowich, Rieger, and Ward \cite{ref3} constructed localized Lagrange bases to bound Sobolev norms by bounded-domain \(L^p\) norms. Because those constructions rely on finite smoothness and spaces distinct from analytic primitive spaces, our main analytical challenge is different: maintaining uniform coefficient stability for multiscale analytic generators as centers reach the endpoints of the interval.

Analytic kernels and multiscale Bernstein estimates each possess rich theories in other geometries: Mhaskar \cite{ref6} treated Gaussian networks on Euclidean space, Li and Cao \cite{ref5} studied multiscale interpolation on the sphere, and Mhaskar, Narcowich, Prestin, and Ward \cite{ref7} established \(L^p\) Bernstein estimates for spherical basis functions. The present work combines positive exponential spectra, multiple coupled scales within a single trial space, an unweighted interval norm, and unrestricted cross-scale center placement up to the boundary.

For a single exponential spectrum, the positive-frequency parts of our derivative atoms coincide with Cauchy–Paul wavelets, whose connections to Bergman spaces \cite{ref1} and nonharmonic Fourier interpolation \cite[Theorem~3]{ref10} provide the natural whole-line background; here we obtain the required whole-line Bessel and Gram bounds directly from explicit inner products. When centers are buffered away from \(\{0,1\}\) by a fixed multiple of their scale, exterior tails decay exponentially with the derivative order. That buffer argument breaks down as centers approach the boundary, where an endpoint atom loses half of its energy to the exterior. Our analysis overcomes this by combining explicit retention of exterior diagonal energy with endpoint-distance weights that remain uniformly summable across scales.

Logistic primitives furnish a primary example of our kernel class. Huybrechs and Trefethen \cite{ref4} linked sigmoid approximation, rational approximation, and logarithmic variable changes, while Min \cite{ref8} proved inverse estimates for rational functions with prescribed poles. Translating rational inverse bounds to a prescribed logistic space requires controlling Jacobian weights and matching pole locations to physical scales; working directly with physical-space primitives avoids these transformations. Constructive sigmoid approximation \cite{ref12} addresses forward approximation rates for specific network operators, whereas we bound derivatives uniformly across an entire multiscale trial space.

The remainder of the paper is organized as follows. Section 2 introduces the spectral class, admissible multiscale configurations, main theorems, and representative kernel examples. Section 3 proves the interval Riesz bounds at high derivative orders and deduces the multiscale inverse inequality. Section 4 establishes sharpness in both scale and derivative order, including the two-atom Gram asymptotics. Section 5 derives a deterministic least-squares sampling stability bound. Section 6 presents numerical experiments validating the boundary-layer asymptotics, multiscale Gram spectra, sharp inverse constants, and sampling stability, and Section 7 summarizes open problems.

\section{Kernels, configurations, and inverse inequalities}

We adopt the Fourier transform convention \(\widehat{K}(\xi) = \int_{\mathbb{R}} e^{-ix\xi} K(x)\,dx\).

\begin{definition}[Positive exponential spectral representation]\label{def:kernel} Let \(K \in C^\infty(\mathbb{R}) \cap L^1(\mathbb{R})\) be real and even. We say that \(K\) has a \emph{positive exponential spectral representation} if there exist \(a > 0\), an integer \(\beta \ge 0\), and a finite or countable family of parameters \(A_j \ge 1\) and weights \(w_j > 0\) (\(j \ge 0\), not necessarily with distinct \(A_j\)), with \(A_0 = 1\), such that

\[
\widehat{K}(\xi) = \vert{}\xi\vert{}^\beta \sum_{j \ge 0} w_j e^{-a A_j \vert{}\xi\vert{}} \quad (\xi \ne 0), \qquad \sum_{j \ge 0} w_j A_j^{-s} < \infty \tag{1}
\]

for some integer \(s \ge 1\) (for a finite family, the sums in (1) run over its finite index set). We define the primitive \(\Sigma(t) = \int_0^t K(y)\,dy\), so that \(\Sigma' = K\) on \(\mathbb{R}\).

\end{definition}

Here positivity refers to the spectral weights \(w_j > 0\). For \(\vert{}\xi\vert{} \ge 1\) and \(A_j \ge 1\), writing \(e^{-a A_j \vert{}\xi\vert{}} \le e^{-a\vert{}\xi\vert{}} e^{-a(A_j - 1)}\) and using \(\sup_{A \ge 1} A^s e^{-a(A-1)} < \infty\), we infer from (1) that \(\vert{}\widehat{K}(\xi)\vert{} \le C \vert{}\xi\vert{}^\beta e^{-a\vert{}\xi\vert{}}\) for \(\vert{}\xi\vert{} \ge 1\). Since \(\widehat{K}\) is also bounded near zero by integrability of \(K\), Fourier inversion extends \(K\) holomorphically to every strip \(\vert{}\operatorname{Im} z\vert{} < a' < a\) and ensures \(K^{(q)} \in L^2(\mathbb{R})\) for all \(q \ge 0\).

Fix separation parameters \(\lambda > 1\) and \(\Delta > 0\). We consider finite multiscale configurations satisfying

\[
0 < h_L < \cdots < h_1 \le 1, \qquad h_{\ell+1} \le h_\ell / \lambda, \qquad X_\ell \subset [0,1], \qquad \vert{}c - d\vert{} \ge \Delta h_\ell \quad (c \ne d \in X_\ell), \tag{2}
\]

with no separation imposed between centers at different scales. Write \(\mathcal{I} = \{(\ell, c) : 1 \le \ell \le L, \, c \in X_\ell\}\) for the configuration index set, set \(h_* = h_L\), and define the multiscale trial space on \(I = (0,1)\) by

\[
V = \operatorname{span}\left\{1, \, x, \, \Sigma\!\left(\frac{x - c}{h_\ell}\right) : (\ell, c) \in \mathcal{I}\right\}\Big\vert{}_I. \tag{3}
\]

Because \(1 \in V\), shifting \(\Sigma\) by an additive constant—for instance, turning \(\Sigma\) into a transition from \(0\) to \(1\) when \(K \ge 0\) and \(\int_{\mathbb{R}} K = 1\)—leaves the space \(V\) unchanged.

\begin{theorem}[Interval inverse inequality]\label{thm:inverse} \emph{Under Definition 2.1 and (2), there exists a constant} \(C = C(K, \lambda, \Delta) < \infty\) \emph{such that}

\[
\left\lVert f' \right\rVert_{L^2(I)} \le C h_*^{-1} \left\lVert f \right\rVert_{L^2(I)} \qquad (f \in V). \tag{4}
\]

\emph{The constant} \(C\) \emph{is independent of} \(L\)\emph{, the cardinalities} \(\# X_\ell\)\emph{, the individual scales, and the distances of the centers from} \(\{0,1\}\)\emph{. More generally, for each integer} \(k \ge 1\)\emph{, there exists} \(C_k = C_k(K, \lambda, \Delta) < \infty\) \emph{such that} \(\left\lVert f^{(k)} \right\rVert_{L^2(I)} \le C_k h_*^{-k} \left\lVert f \right\rVert_{L^2(I)}\) \emph{for all} \(f \in V\)\emph{. Furthermore, the generators in (3) are linearly independent on} \(I\)\emph{, so} \(\dim V = 2 + \#\mathcal{I}\)\emph{.}

\end{theorem}

Because \(C\) depends only on \((K, \lambda, \Delta)\), estimate (4) persists under deleting centers or perturbing them while preserving within-scale separation. Affine rescaling yields the corresponding bound on any interval of length \(b > 0\) provided \(h_1 \le b\), and when \(\mathcal{I} = \varnothing\) estimate (4) reduces to the classical Markov–Bernstein inequality for affine functions.

Differentiating \(f \in V\) to a high order \(q + 1\) annihilates the affine subspace \(\operatorname{span}\{1, x\}\) and converts the globally supported primitives \(\Sigma((x - c)/h_\ell)\) into localized oscillatory derivative atoms. To formulate stability at order \(q\), define the \(L^2(\mathbb{R})\)-normalized atoms

\[
N_q = \left\lVert K^{(q)} \right\rVert_{L^2(\mathbb{R})}^2, \qquad \Psi_q = N_q^{-1/2} K^{(q)}, \qquad \psi_{h,c}^{(q)}(x) = h^{-1/2} \Psi_q\!\left(\frac{x - c}{h}\right). \tag{5}
\]

\begin{theorem}[Interval coefficient bounds]\label{thm:gram} \emph{For every} \(0 < \zeta < 1/2\)\emph{, there exists an integer} \(q_0 = q_0(K, \lambda, \Delta, \zeta)\) \emph{such that, for all} \(q \ge q_0\)\emph{, every finite configuration satisfying (2), and every real family} \((b_{\ell c})_{(\ell,c) \in \mathcal{I}}\)\emph{,}

\[
\left(\frac{1}{2} - \zeta\right) \sum_{(\ell, c) \in \mathcal{I}} b_{\ell c}^2 \le \left\lVert \sum_{(\ell, c) \in \mathcal{I}} b_{\ell c} \psi_{h_\ell, c}^{(q)} \right\rVert_{L^2(I)}^2 \le (1 + \zeta) \sum_{(\ell, c) \in \mathcal{I}} b_{\ell c}^2. \tag{6}
\]

The limiting lower constant \(1/2\) cannot be improved: since \(\vert{}K^{(q)}\vert{}^2\) is even, an atom centered at \(0\) or \(1\) places at most half of its whole-line energy in \(I\). For a single exponential spectrum, Corollary 4.4 establishes the sharp \(O(q^{-1/2})\) rate at which the best uniform lower constant approaches \(1/2\).

\subsection{Examples and scope of the spectral assumption}

\end{theorem}

\begin{example}[Logistic derivative kernel]\label{ex:logistic} Let \(\sigma(t) = (1 + e^{-t})^{-1}\). For fixed \(\gamma > 0\), set \(K_\gamma(t) = \gamma e^{-\gamma t} / (1 + e^{-\gamma t})^2\) and \(a = \pi / \gamma\). Standard contour integration of \(e^{-i\xi z} e^z / (1 + e^z)^2\) on the strip \(0 \le \operatorname{Im} z \le 2\pi\) (enclosing the double pole at \(z = i\pi\) with residue \(i\xi e^{\pi\xi}\)) gives \(\widehat{K}_1(\xi) = \pi\xi / \sinh(\pi\xi)\); by evenness and scaling, \(\widehat{K}_\gamma(0) = 1\) and

\[
\widehat{K}_\gamma(\xi) = \frac{a \xi}{\sinh(a \xi)} = 2 a \vert{}\xi\vert{} \sum_{j \ge 0} e^{-a(2j+1)\vert{}\xi\vert{}}, \qquad \xi \ne 0. \tag{7}
\]

Hence \(K_\gamma\) satisfies Definition 2.1 with \(\beta = 1\), \(A_j = 2j + 1\), \(w_j = 2a\), and \(s = 2\), and its primitive is \(\Sigma(t) = \sigma(\gamma t) - 1/2\).

\end{example}

\begin{example}[Cauchy kernel]\label{ex:cauchy} For \(K_a(t) = a / [\pi (a^2 + t^2)]\) (\(a > 0\)), the Fourier cosine integral

\[
\pi^{-1} \int_0^\infty e^{-a\xi} \cos(t\xi)\,d\xi = K_a(t)
\]

yields \(\widehat{K}_a(\xi) = e^{-a\vert{}\xi\vert{}}\), which satisfies Definition 2.1 with \(\beta = 0\) and a single term \(A_0 = w_0 = 1\), with associated transition function \(1/2 + \pi^{-1} \arctan(t/a)\). Thus Definition 2.1 covers kernels with slow algebraic decay in physical space.

\end{example}

\begin{proposition}[Finite positive mixtures]\label{prop:finite-mixture} \emph{If} \(K_1, \dots, K_m\) \emph{satisfy Definition 2.1 with a common exponent} \(\beta\)\emph{, then every positive linear combination} \(K = \sum_{i=1}^m \omega_i K_i\) \emph{(}\(\omega_i > 0\)\emph{) also satisfies Definition 2.1.}

\end{proposition}

\begin{proof} Letting \(a_i > 0\) denote the minimal decay parameter of \(K_i\) and setting \(a = \min_{1 \le i \le m} a_i\), we obtain a representation of the form (1) with rescaled exponents \(A'_{ij} = a_i A_{ij} / a \ge 1\) (attaining \(1\) whenever \(a_i = a\) and \(A_{i0} = 1\)) and positive weights \(\omega_i w_{ij}\). Choosing \(s = \max_{1 \le i \le m} s_i\) preserves the summability condition in (1), while smoothness, evenness, and integrability carry over directly to the finite sum. \end{proof}

Proposition 2.6 encompasses finite positive mixtures of Cauchy kernels (\(\beta = 0\)) and, separately, of logistic derivatives (\(\beta = 1\)). For countable mixtures, smoothness, integrability, attainment of a minimal decay parameter \(a > 0\), and the inverse-power summability in (1) must be verified directly; combining kernels with distinct powers \(\beta\) likewise falls outside Definition 2.1 unless rewritten with a common prefactor \(\vert{}\xi\vert{}^\beta\). Note also that analyticity alone does not suffice: a Gaussian spectrum \(e^{-c\xi^2}\) decays faster than every exponential term \(\vert{}\xi\vert{}^\beta e^{-a\vert{}\xi\vert{}}\), while signed spectral mixtures may suffer high-frequency cancellation.

\begin{example}[Dyadic multiscale logistic space]\label{ex:dyadic} Fix \(\gamma > 0\) and consider the dyadic space on \(I = (0,1)\) given by

\[
V_J = \operatorname{span}\biggl\{1, \, x, \, \sigma\!\left(\frac{\gamma(x - c_{\ell r})}{h_\ell}\right) : h_\ell = 2^{-\ell}, \quad c_{\ell r} = (2r - 1)h_\ell, \quad 1 \le \ell \le J, \quad 1 \le r \le 2^{\ell-1}\biggr\}. \tag{8}
\]

Since \(1 \in V_J\) absorbs the constant \(-1/2\) in \(\Sigma(t) = \sigma(\gamma t) - 1/2\), this space has the form (3) with \(\lambda = \Delta = 2\) and \(h_* = 2^{-J}\). Theorem 2.2 gives \(L(V_J) \lesssim_\gamma 2^J\) and \(\dim V_J = 2^J + 1\), and Corollary 4.2 supplies the matching lower bound; here \(h_*^{-1}\) is comparable to \(\dim V_J\), in contrast to general sparse configurations.

\end{example}

\section{One-sided stability and the inverse inequality}

We first establish the interval Riesz bounds for a single exponential spectrum via the decomposition \(G_I = G_{\mathbb{R}} - E_- - E_+\), controlling the eigenvalue deficit \(\delta_n = \tau_n + \alpha_n + 2\rho_n\) through three quantities: the far-endpoint diagonal tail loss \(\tau_n\), the whole-line off-diagonal row-sum bound \(\alpha_n\), and the exterior half-line off-diagonal row-sum bound \(\rho_n\). We then extend these bounds to positive spectral mixtures by operator-norm perturbation and apply Sobolev interpolation to deduce Theorem 2.2.

\subsection{Single-exponential atoms and endpoint phases}

Consider the single-exponential kernel \(\widehat{K}_1(\xi) = \vert{}\xi\vert{}^\beta e^{-a\vert{}\xi\vert{}}\), set \(n = q + \beta + 1\), and let \(J_n = \int_{\mathbb{R}} (1 + t^2)^{-n}\,dt\). By Fourier inversion,

\[
K_1^{(q)}(x) = \frac{\Gamma(n)}{\pi} \operatorname{Re}\left\{i^q (a - ix)^{-n}\right\}, \qquad F_{h,c}(x) = \frac{(1 - i(x - c)/(ah))^{-n}}{\sqrt{a h J_n}}, \qquad \psi_{h,c} = \sqrt{2}\,\operatorname{Re}(i^q F_{h,c}), \tag{9}
\]

where we suppress the superscript \(q\) on single-exponential atoms throughout Sections 3.1–3.2. By definition of \(J_n\), \(\left\lVert F_{h,c} \right\rVert_{L^2(\mathbb{R})} = 1\). Since \((1 - it)^{-n}\) is holomorphic in the lower half-plane \(\operatorname{Im} t < 0\), \(\widehat{F}_{h,c}\) is supported in \((0, \infty)\); consequently \(F_{h,c}\) and \(\overline{F_{h,c}}\) are orthogonal in \(L^2(\mathbb{R})\), which gives \(\left\lVert \psi_{h,c} \right\rVert_{L^2(\mathbb{R})}^2 = 1\) and the parity relation \(\psi_{h,c}(c - x) = (-1)^q \psi_{h,c}(c + x)\). Evaluating \(J_n = B(1/2, n - 1/2)\) via the beta–gamma relation and Stirling's ratio expansion yields, for \(n \ge 8\),

\[
J_n = \sqrt{\pi}\,\frac{\Gamma(n - 1/2)}{\Gamma(n)} = \sqrt{\frac{\pi}{n}}\bigl(1 + O(n^{-1})\bigr), \qquad c_0 n^{-1/2} \le J_n \le C_0 n^{-1/2}. \tag{10}
\]

\begin{lemma}[Cross-scale estimate with endpoint-distance weights]\label{lem:crossscale} \emph{Let} \(n = q + \beta + 1 \ge 8\) \emph{and} \(w(u) = (1 + u^2)^{-1}\)\emph{. If} \(0 < h \le H / \lambda\) \emph{and} \(c, d \ge 0\)\emph{, then}

\[
\left\lvert \int_{-\infty}^0 \psi_{h,c}(x) \psi_{H,d}(x)\,dx \right\rvert \le e_n \sqrt{h/H}\,w(c/(ah)) w(d/(aH)), \tag{11}
\]

\emph{where, with} \(\varepsilon = \sqrt{(\lambda - 1)/2}\)\emph{, one may take}

\[
e_n = C_\lambda \bigl\{n^{-1/2} + (1 + \varepsilon^2)^{-(n-2)/4}\bigr\}. \tag{12}
\]

\end{lemma}

\begin{proof} Set \(r = h/H \le \lambda^{-1}\), \(u_0 = c/(ah)\), and \(v_0 = d/(aH)\). Substituting \(x = -y\) (\(y > 0\)) on \((-\infty, 0)\) and setting \(u = (y + c)/(ah)\) and \(v = (y + d)/(aH) = v_0 + r(u - u_0)\), we write \(1 + iu = \sqrt{1 + u^2}\,e^{i\arctan u}\) and apply \(\operatorname{Re}(z_1)\operatorname{Re}(z_2) = \frac{1}{2}\operatorname{Re}(z_1\overline{z_2} + z_1 z_2)\) (noting \(\operatorname{Re}\overline{z} = \operatorname{Re} z\)) to obtain

\[
\int_{-\infty}^0 \psi_{h,c}(x) \psi_{H,d}(x)\,dx = \frac{\sqrt{r}}{J_n} \operatorname{Re} \int_{u_0}^\infty A(u)\bigl(e^{in\Phi_-(u)} + (-1)^q e^{in\Phi_+(u)}\bigr)\,du,
\]

with amplitude and phases

\[
A(u) = (1 + u^2)^{-n/2} (1 + v^2)^{-n/2}, \qquad \Phi_\pm(u) = \arctan u \pm \arctan v.
\]

When \(u_0 < \varepsilon\), we have on \([u_0, \varepsilon]\) the uniform phase bounds

\[
\vert{}\Phi'_\pm\vert{} \ge \kappa := \frac{1}{1 + \varepsilon^2} - \lambda^{-1} = \frac{\lambda - 1}{\lambda(\lambda + 1)} > 0, \qquad \vert{}\Phi''_\pm\vert{} \le 2.
\]

Integrating by parts on \([u_0, \varepsilon]\) gives

\[
\int_{u_0}^\varepsilon A(u) e^{in\Phi_\pm(u)}\,du = \left[\frac{A(u)}{in\Phi'_\pm(u)} e^{in\Phi_\pm(u)}\right]_{u_0}^\varepsilon - \frac{1}{in} \int_{u_0}^\varepsilon \left(\frac{A'(u)}{\Phi'_\pm(u)} - \frac{A(u)\Phi''_\pm(u)}{\Phi'_\pm(u)^2}\right) e^{in\Phi_\pm(u)}\,du.
\]

Since \(u, v \ge 0\), \(A\) is decreasing on \([u_0, \varepsilon]\); thus the boundary terms are bounded by \(2A(u_0)/(n\kappa)\) and \(\int_{u_0}^\varepsilon \vert{}A'\vert{} \le A(u_0)\), yielding

\[
\left\lvert \int_{u_0}^\varepsilon A(u) e^{in\Phi_\pm(u)}\,du \right\rvert \le \frac{A(u_0)}{n} \left(\frac{3}{\kappa} + \frac{2\varepsilon}{\kappa^2}\right). \tag{13}
\]

Since \(A(u_0) \le w(u_0) w(v_0)\), dividing by \(J_n\) and using (10) produces the \(n^{-1/2}\) term in (12).

On the remaining ray \(u \ge U := \max(u_0, \varepsilon)\), monotonicity gives \((1 + u^2)^{-1} \le w(u_0)\), \((1 + v^2)^{-1} \le w(v_0)\), and \((1 + v^2)^{-(n-2)/2} \le 1\). Splitting \((1 + u^2)^{-(n-2)/2}\) into two equal factors and bounding the first at \(u \ge \varepsilon\), we find

\[
\begin{aligned}
\int_U^\infty A(u)\,du &\le w(u_0) w(v_0) (1 + \varepsilon^2)^{-(n-2)/4} \int_0^\infty (1 + u^2)^{-(n-2)/4}\,du \\
&\le C n^{-1/2} (1 + \varepsilon^2)^{-(n-2)/4} w(u_0) w(v_0).
\end{aligned}
\]

Multiplying by \(2\sqrt{r}/J_n\) and invoking (10) completes the proof. \end{proof}

\subsection{Center counting and single-exponential Riesz bounds}

We now assemble the exterior and whole-line Gram bounds for a configuration satisfying (2).

\textbf{Same-scale exterior interactions.} For distinct centers at the same scale, assume \(0 \le c < d \le 1\), so that \(v_0 = d/(ah) \ge (d - c)/(ah) \ge \Delta / a\). With \(u = (y + c)/(ah)\) and \(v = (y + d)/(ah)\) on \((-\infty, 0)\) (\(y > 0\)), we have \((1 + u^2)^{-1} \le w(u_0)\), \((1 + v^2)^{-1} \le w(v_0)\), and \((1 + v^2)^{-(n-2)/2} \le (1 + (\Delta/a)^2)^{-(n-2)/2}\). Integrating the remaining factor \((1 + u^2)^{-(n-2)/2}\) over \((0, \infty)\) contributes \(O(n^{-1/2})\), which cancels \(J_n\) in (10) and gives, for \(n \ge 8\),

\[
\left\lvert \int_{-\infty}^0 \psi_{h,c} \psi_{h,d} \right\rvert \le b_n w(c/(ah)) w(d/(ah)), \qquad b_n = C\bigl(1 + (\Delta/a)^2\bigr)^{-(n-2)/4}. \tag{14}
\]

\textbf{Summability of endpoint-distance weights.} Within-scale separation \(\vert{}c - d\vert{} \ge \Delta h_\ell\) and integral comparison imply

\[
\sum_{c \in X_\ell} w(c/(ah_\ell)), \quad \sum_{c \in X_\ell} w((1 - c)/(ah_\ell)) \le P_{a, \Delta} := 2 + \frac{\pi a}{\Delta}. \tag{15}
\]

Let \(E_-\) denote the exterior Gram matrix on \((-\infty, 0)\). Summing Lemma 3.1 over both coarser and finer scales via \(\sum_{k \ge 1} \lambda^{-k/2} = (\sqrt{\lambda} - 1)^{-1}\) and combining with (14) and (15), we bound every off-diagonal row sum of \(E_-\) by

\[
\rho_n = P_{a, \Delta} \left(b_n + \frac{2 e_n}{\sqrt{\lambda} - 1}\right) \longrightarrow 0 \qquad (n \to \infty). \tag{16}
\]

By reflection \(x \mapsto 1 - x\), the right exterior Gram matrix \(E_+\) on \((1, \infty)\) satisfies the same off-diagonal row-sum bound.

\textbf{Whole-line off-diagonal decay.} From (9), the Fourier transform of \(F_{h,c}\) is

\[
\widehat{F}_{h,c}(\xi) = \frac{2\pi (ah)^{n-1/2}}{\Gamma(n)\sqrt{J_n}}\,\xi^{n-1} e^{-(ah + ic)\xi} \mathbf{1}_{(0,\infty)}(\xi).
\]

Because \(F_{h,c}\) and \(\overline{F_{H,d}}\) have disjoint frequency supports, \(\langle \psi_{h,c}, \psi_{H,d} \rangle_{L^2(\mathbb{R})} = \operatorname{Re}\langle F_{h,c}, F_{H,d} \rangle_{L^2(\mathbb{R})}\). Applying Plancherel's theorem together with \(J_n = \sqrt{\pi}\,\Gamma(n - 1/2)/\Gamma(n)\) and the Legendre duplication formula \(\Gamma(2n - 1) = \pi^{-1/2} 2^{2n-2} \Gamma(n - 1/2)\Gamma(n)\) gives

\[
\left\lvert \int_{\mathbb{R}} \psi_{h,c}(x)\psi_{H,d}(x)\,dx \right\rvert \le \left\lvert \int_{\mathbb{R}} F_{h,c}(x)\overline{F_{H,d}(x)}\,dx \right\rvert = g_{h,c; H,d}^{\,2n-1},
\]

where

\[
g_{h,c; H,d} = \frac{2\sqrt{h H}}{\sqrt{(h + H)^2 + ((c - d)/a)^2}}.
\]

For any pair of distinct atoms in \(\mathcal{I}\), \(g_{h,c; H,d} \le \theta < 1\) with

\[
\theta = \max\left\{\bigl(1 + (\Delta/(2a))^2\bigr)^{-1/2}, \, \frac{2\sqrt{\lambda}}{1 + \lambda}\right\}.
\]

Since \(X_H\) is \(\Delta H\)-separated and \(t \mapsto (1 + t^2)^{-3/2}\) is unimodal, integral comparison bounds the cubic sum over \(X_H\) by its two largest terms plus the corresponding integral:

\[
\sum_{d \in X_H} g_{h,c; H,d}^{\,3} \le \left(\frac{2\sqrt{h H}}{h + H}\right)^3 \left(2 + \frac{2a(h + H)}{\Delta H}\right). \tag{17}
\]

The right-hand side of (17) is \(O_{a, \Delta}(R^{-3/2})\) when \(R = H/h \ge 1\) and \(O_{a, \Delta}(R^{-1/2})\) when \(R = h/H \ge 1\). Factoring \(g^{2n-1} \le \theta^{2n-4} g^3\) for off-diagonal entries and summing geometrically over scales bounds every off-diagonal row sum of the whole-line Gram matrix \(G_{\mathbb{R}}\) by

\[
\alpha_n = C_{a, \lambda, \Delta} \theta^{2n-4} \longrightarrow 0 \qquad (n \to \infty). \tag{18}
\]

\textbf{Diagonal interval energy.} Every center \(c \in [0,1]\) lies at distance at least \(1/2\) from at least one endpoint of \(I\). By evenness of \(\psi_{h,c}^2\) about \(c\) and the condition \(h \le 1\), each diagonal entry of \(G_I\) satisfies

\[
\frac{1}{2} - \tau_n \le D_{h,c} := \int_0^1 \psi_{h,c}^2 \le 1, \qquad \tau_n = \frac{2}{J_n} \int_{1/(2a)}^\infty (1 + t^2)^{-n}\,dt, \tag{19}
\]

where \(\tau_n\) decays exponentially in \(n\) since \((1 + t^2)^{-n} \le (1 + (2a)^{-2})^{-(n-1)/2} (1 + t^2)^{-(n+1)/2}\) for \(t \ge 1/(2a)\).

Let \(U_q : \ell^2(\mathcal{I}) \to L^2(\mathbb{R})\) denote the synthesis operator \(U_q b = \sum_{(\ell,c) \in \mathcal{I}} b_{\ell c} \psi_{h_\ell, c}\) for the single-exponential atoms. Combining the preceding estimates yields uniform Riesz bounds on \(I\).

\begin{proposition}[Single-exponential interval Riesz bounds]\label{prop:single-exp} \emph{For every} \(n = q + \beta + 1 \ge 8\)\emph{, every configuration satisfying (2), and every} \(b \in \ell^2(\mathcal{I})\)\emph{,}

\[
\left(\frac{1}{2} - \delta_n\right) \left\lVert b \right\rVert_{\ell^2(\mathcal{I})}^2 \le \left\lVert U_q b \right\rVert_{L^2(I)}^2 \le (1 + \alpha_n) \left\lVert b \right\rVert_{\ell^2(\mathcal{I})}^2, \qquad \delta_n = \tau_n + \alpha_n + 2\rho_n = O_{a, \lambda, \Delta}(n^{-1/2}). \tag{20}
\]

\end{proposition}

\begin{proof} From \(G_I = G_{\mathbb{R}} - E_- - E_+\) together with (16) and (18), every off-diagonal row sum of the symmetric matrix \(G_I\) is bounded by \(\alpha_n + 2\rho_n\), while its diagonal entries satisfy (19). Schur's test gives the lower bound in (20). On \(\mathbb{R}\), \(G_{\mathbb{R}}\) has unit diagonal and off-diagonal row sums bounded by \(\alpha_n\), so \(\left\lVert U_q b \right\rVert_{L^2(I)}^2 \le \left\lVert U_q b \right\rVert_{L^2(\mathbb{R})}^2 \le (1 + \alpha_n)\left\lVert b \right\rVert_{\ell^2(\mathcal{I})}^2\). Finally, (12), (14), (16), (18), and (19) yield \(\delta_n = O_{a,\lambda,\Delta}(n^{-1/2})\). \end{proof}

\subsection{Positive mixtures in synthesis-operator norm}

\begin{proof}[Proof of Theorem 2.3] For \(A \ge 1\), let \(\widehat{K}_A(\xi) = \vert{}\xi\vert{}^\beta e^{-a A \vert{}\xi\vert{}}\), and let \(U_{q,A} : \ell^2(\mathcal{I}) \to L^2(\mathbb{R})\) be the dilated single-exponential synthesis operator \(U_{q,A} b = \sum_{(\ell,c) \in \mathcal{I}} b_{\ell c} \psi_{A h_\ell, c}\), so that \(U_{q,1} = U_q\). Although dilation by \(A\) reduces the relative within-scale center separation, \(U_{q,A}\) satisfies a whole-line Bessel bound (requiring no upper bound on \(A h_\ell\)) with linear growth in \(A\). Indeed, replacing \(a\) by \(a A\) in the off-diagonal same-scale sum gives

\[
2 \sum_{k \ge 1} \left(1 + \left(\frac{k \Delta}{2 a A}\right)^2\right)^{-3/2} \le \frac{4 a A}{\Delta},
\]

while for distinct scales with ratio \(R \ge 1\), (17) with \(a A\) in place of \(a\) gives

\[
\sum_{d \in X_H} g_A^{\,3} \le \begin{cases} 16(1 + a A / \Delta) R^{-3/2}, & H = R h, \\[2pt] 16 R^{-3/2} + 16(a A / \Delta) R^{-1/2}, & h = R H, \end{cases}
\]

where \(g_A\) denotes the whole-line Gram factor for \(K_A\). Since \(g_A^{2n-1} \le g_A^3\) whenever \(q + \beta \ge 1\), summing over scales yields the configuration-uniform whole-line bound

\[
\left\lVert U_{q,A} \right\rVert_{\ell^2(\mathcal{I}) \to L^2(\mathbb{R})}^2 \le 1 + \frac{4 a A}{\Delta} + \frac{32}{\lambda^{3/2} - 1} + \frac{32 a A}{\Delta(\sqrt{\lambda} - 1)} \le C_{\mathrm{B}}^2 A. \tag{21}
\]

Set \(\nu_q = \left\lVert K_1^{(q)} \right\rVert_{L^2(\mathbb{R})}\) and collect the total weight at the minimal exponent \(A_j = 1\) as \(W := \sum_{A_j = 1} w_j \in (0, \infty)\). By Fourier scaling, \(K_A^{(q)}(x) = A^{-n} K_1^{(q)}(x/A)\), so

\[
h^{-1/2} K_A^{(q)}\!\left(\frac{x - c}{h}\right) = A^{-n + 1/2} (A h)^{-1/2} K_1^{(q)}\!\left(\frac{x - c}{A h}\right).
\]

Relative to \(\nu_q\), the unnormalized mixture synthesis operator therefore decomposes as

\[
\widetilde{T}_q = \sum_{j \ge 0} w_j A_j^{-n + 1/2} U_{q, A_j} = W U_{q,1} + \mathcal{R}_q, \qquad \left\lVert \mathcal{R}_q \right\rVert_{\ell^2(\mathcal{I}) \to L^2(\mathbb{R})} \le C_{\mathrm{B}} \eta_q, \qquad \eta_q = \sum_{A_j > 1} w_j A_j^{-(n - 1)}.
\]

For \(n - 1 \ge s\), the series converges absolutely in operator norm by (1) and (21), and dominated convergence gives \(\eta_q \to 0\) as \(q \to \infty\) without requiring a spectral gap above \(A_0 = 1\) (though the decay rate of \(\eta_q\) depends on how closely \(\{A_j\}_{A_j > 1}\) clusters near \(1\)).

Evaluating \(\widetilde{T}_q\) on a single basis vector shows that \(r_q := \left\lVert K^{(q)} \right\rVert_{L^2(\mathbb{R})} / \nu_q\) satisfies \(\vert{}r_q - W\vert{} \le \eta_q\). For the normalized mixture synthesis operator \(T_q = \widetilde{T}_q / r_q\), whenever \(\eta_q < W\) we obtain

\[
d_q := \left\lVert T_q - U_{q,1} \right\rVert_{\ell^2(\mathcal{I}) \to L^2(\mathbb{R})} \le \frac{2 C_{\mathrm{B}} \eta_q}{W - \eta_q} \longrightarrow 0 \qquad (q \to \infty). \tag{22}
\]

Since restriction from \(\mathbb{R}\) to \(I\) is contractive, Proposition 3.2 and the triangle inequality imply that whenever \(\delta_n < 1/2\) and \(d_q < \sqrt{1/2 - \delta_n}\),

\[
\left(\sqrt{\frac{1}{2} - \delta_n} - d_q\right)^2 \left\lVert b \right\rVert_{\ell^2(\mathcal{I})}^2 \le \left\lVert T_q b \right\rVert_{L^2(I)}^2 \le \left(\sqrt{1 + \alpha_n} + d_q\right)^2 \left\lVert b \right\rVert_{\ell^2(\mathcal{I})}^2. \tag{23}
\]

Given \(0 < \zeta < 1/2\), because \(\alpha_n\), \(\delta_n\), and \(\eta_q\) depend only on \((K, \lambda, \Delta, q)\) and tend to \(0\) as \(q \to \infty\), we may choose \(q_0 = q_0(K, \lambda, \Delta, \zeta)\) sufficiently large so that for every \(q \ge q_0\),

\[
n - 1 \ge s, \qquad n \ge 8, \qquad \eta_q \le \frac{W}{2}, \qquad \delta_n \le \frac{\zeta}{4}, \qquad \alpha_n \le \frac{\zeta}{4}, \qquad d_q \le \frac{\zeta}{8}. \tag{24}
\]

Substituting (24) into (23) directly yields (6). \end{proof}

\subsection{Order reduction via Sobolev interpolation}

\begin{lemma}[Parameterized interval interpolation]\label{lem:interpolation} \emph{For each integer} \(m \ge 1\)\emph{, there exists} \(\mu_m \ge 1\) \emph{such that, for every} \(v \in H^{m+1}(0,1)\)\emph{, integer} \(N \ge 1\)\emph{, and} \(0 \le k \le m\)\emph{,}

\[
\left\lVert v^{(k)} \right\rVert_2 \le \mu_m \bigl(N^k \left\lVert v \right\rVert_2 + N^{k - m - 1} \left\lVert v^{(m+1)} \right\rVert_2\bigr). \tag{25}
\]

\end{lemma}

\begin{proof} For \(N = 1\), write \(R(x) = \int_0^x (x - t)^m v^{(m+1)}(t)\,dt / m!\), so that \(P = v - R\) is a polynomial of degree at most \(m\) and \(\left\lVert R^{(k)} \right\rVert_2 \le C'_m \left\lVert v^{(m+1)} \right\rVert_2\). Norm equivalence on finite-dimensional polynomial spaces gives \(\left\lVert P^{(k)} \right\rVert_2 \le C''_m \left\lVert P \right\rVert_2 \le C''_m(\left\lVert v \right\rVert_2 + \left\lVert R \right\rVert_2)\), establishing (25) for \(N = 1\). For general \(N \ge 1\), apply the \(N = 1\) bound on each subinterval \(I_j = ((j - 1)/N, j/N)\) (\(1 \le j \le N\)) to the rescaled function \(u_j(t) = v((j - 1 + t)/N)\), use \(\left\lVert v^{(k)} \right\rVert_{L^2(I_j)}^2 = N^{2k - 1} \left\lVert u_j^{(k)} \right\rVert_{L^2(0,1)}^2\), and sum over \(j\). \end{proof}

\begin{proposition}[Order reduction from high derivatives]\label{prop:order-reduction} \emph{Let} \(m \ge 1\) \emph{be an integer,} \(f \in H^{m+1}(0,1)\)\emph{,} \(0 < h \le 1\)\emph{, and} \(M \ge 0\)\emph{. If} \(\left\lVert f^{(m+1)} \right\rVert_2 \le M h^{-1} \left\lVert f^{(m)} \right\rVert_2\)\emph{, then}

\[
\left\lVert f^{(k)} \right\rVert_2 \le C_{m,M} h^{-k} \left\lVert f \right\rVert_2 \qquad (1 \le k \le m + 1).
\]

\end{proposition}

\begin{proof} Taking \(k = m\) and \(N = \lceil (1 + 2 \mu_m M) / h \rceil\) in (25), we absorb the highest-order term \(\mu_m N^{-1} \left\lVert f^{(m+1)} \right\rVert_2 \le \frac{1}{2}\left\lVert f^{(m)} \right\rVert_2\) into the left-hand side to obtain \(\left\lVert f^{(m)} \right\rVert_2 \le \widetilde{C}_{m,M} h^{-m} \left\lVert f \right\rVert_2\) with \(\widetilde{C}_{m,M} = 2 \mu_m (2 + 2 \mu_m M)^m\). The hypothesis then yields the bound for \(k = m + 1\), and choosing \(N = \lceil 1/h \rceil\) in (25) gives \(\left\lVert f^{(k)} \right\rVert_2 \le \mu_m (2^k + M \widetilde{C}_{m,M}) h^{-k} \left\lVert f \right\rVert_2\) for \(1 \le k < m\). Note that because (25) imposes no boundary conditions, polynomial components annihilated by \(f^{(m)}\) are automatically controlled. \end{proof}

\begin{proof}[Proof of Theorem 2.2] Fix \(\zeta = 1/4\) and an integer \(q \ge \max\{1, q_0(K, \lambda, \Delta, 1/4)\}\) depending only on \((K, \lambda, \Delta)\). For any \(f(x) = a_0 + a_1 x + \sum_{(\ell, c) \in \mathcal{I}} d_{\ell c} \Sigma((x - c)/h_\ell) \in V\), differentiating \(q + 1\) times gives \(f^{(q+1)} = N_q^{1/2} \sum_{(\ell, c) \in \mathcal{I}} (d_{\ell c} h_\ell^{-q - 1/2}) \psi_{h_\ell, c}^{(q)}\) (and similarly at order \(q + 2\)). Applying Theorem 2.3 at orders \(q\) and \(q + 1\) with \(c_1 = 1/4\) and \(c_2 = 5/4\) yields

\[
\begin{aligned}
\left\lVert f^{(q+1)} \right\rVert_2^2 &\ge c_1 N_q \sum_{(\ell, c) \in \mathcal{I}} d_{\ell c}^2 h_\ell^{-2q - 1}, \\
\left\lVert f^{(q+2)} \right\rVert_2^2 &\le c_2 N_{q+1} \sum_{(\ell, c) \in \mathcal{I}} d_{\ell c}^2 h_\ell^{-2q - 3} \le \frac{c_2 N_{q+1}}{c_1 N_q} h_*^{-2} \left\lVert f^{(q+1)} \right\rVert_2^2.
\end{aligned} \tag{26}
\]

Applying Proposition 3.4 with \(m = q + 1\) establishes (4), and choosing \(q\) so that \(q + 2 \ge k\) gives the \(k\)th-order bound. Finally, if \(f = 0\) on \(I\), the first inequality in (26) forces \(d_{\ell c} = 0\) for all \((\ell, c) \in \mathcal{I}\), whence \(a_0 = a_1 = 0\) by linear independence of \(\{1, x\}\) on \(I\). \end{proof}

Because this reduction passes through derivatives of order \(q + 2\) and polynomial-remainder interpolation, the resulting constants \(C\) and \(C_k\) are computable in principle but not sharp (see Section 6.3 and Figure 4(a) for empirical values of \(h_* L(V)\)).

\section{Sharpness in scale and derivative order}

\subsection{Scale sharpness for each fixed kernel}

\begin{theorem}[Matching scale lower bound]\label{thm:scale-sharp} \emph{Fix a kernel} \(K\) \emph{satisfying Definition 2.1 and parameters} \(\lambda > 1\)\emph{,} \(\Delta > 0\)\emph{. There exist single-scale, two-center spaces} \(V_h\) \emph{satisfying (2) such that, as} \(h \downarrow 0\)\emph{,}

\[
c_{K, \Delta} h^{-1} \le L(V_h) \le C_{K, \lambda, \Delta} h^{-1}, \qquad \dim V_h = 4. \tag{27}
\]

\emph{Thus the scaling order} \(h_*^{-1}\) \emph{is optimal over admissible configurations for every fixed kernel in the class.}

\end{theorem}

\begin{proof} Set \(d = \Delta\), \(b = 1/2\), and define the local difference

\[
Q_d(t) = \Sigma(t + d/2) - \Sigma(t - d/2) = \int_{-d/2}^{d/2} K(t - y)\,dy.
\]

By Young's inequality, \(\left\lVert Q_d \right\rVert_{L^2(\mathbb{R})} \le \sqrt{d}\,\left\lVert K \right\rVert_{L^1(\mathbb{R})} < \infty\), and \(Q'_d = K(\cdot + d/2) - K(\cdot - d/2) \in L^2(\mathbb{R})\). Since \(\widehat{K}(\xi) > 0\) for \(\xi \ne 0\), \(K\) is nonzero; as a nonzero integrable function cannot be \(d\)-periodic, \(Q'_d \not\equiv 0\), so both \(\left\lVert Q_d \right\rVert_{L^2(\mathbb{R})}\) and \(\left\lVert Q'_d \right\rVert_{L^2(\mathbb{R})}\) are strictly positive.

For \(0 < h \le \min\{1, (2\Delta)^{-1}\}\), the single-scale configuration with centers \(b \pm d h / 2 \in [1/4, 3/4]\) satisfies (2), and the corresponding four-dimensional space \(V_h\) contains \(p_h(x) = Q_d((x - b)/h)\). Setting \(I_h = (-b/h, (1 - b)/h)\), a change of variables gives

\[
h \frac{\left\lVert p'_h \right\rVert_{L^2(I)}}{\left\lVert p_h \right\rVert_{L^2(I)}} = \left(\frac{\int_{I_h} \vert{}Q'_d(t)\vert{}^2\,dt}{\int_{I_h} \vert{}Q_d(t)\vert{}^2\,dt}\right)^{1/2} \longrightarrow \frac{\left\lVert Q'_d \right\rVert_{L^2(\mathbb{R})}}{\left\lVert Q_d \right\rVert_{L^2(\mathbb{R})}} > 0 \qquad (h \downarrow 0), \tag{28}
\]

which proves the lower bound in (27), while the upper bound and dimension follow from Theorem 2.2. Because both centers stay uniformly inside \(I\), this \(h^{-1}\) growth arises purely from fine-scale pulse formation rather than boundary effects. \end{proof}

Theorem 4.1 shows that \(L(V_h)\) can diverge at order \(h_*^{-1}\) while \(\dim V_h = 4\) remains fixed. For the dyadic logistic space \(V_J\) of Example 2.7 (8), setting

\[
Q_\gamma(t) = \sigma(\gamma(t + 1)) - \sigma(\gamma(t - 1)) = \frac{\sinh \gamma}{\cosh \gamma + \cosh(\gamma t)}
\]

gives \(0 < Q_\gamma < 1\), \(\int_{\mathbb{R}} Q_\gamma = 2\) (hence \(\int_{\mathbb{R}} Q_\gamma^2 \le 2\)), and \(I_\gamma := \int_{-1}^1 \vert{}Q'_\gamma\vert{}^2 > 0\). For \(J \ge 2\), \(1/2 \pm h_*\) belong to \(X_J\), so \(p_J(x) = Q_\gamma((x - 1/2)/h_*) \in V_J\) satisfies

\[
\left\lVert p_J \right\rVert_2^2 \le 2 h_*, \qquad \left\lVert p'_J \right\rVert_2^2 \ge I_\gamma / h_*.
\]

Testing with \(x - 1/2 \in V_1\) for \(J = 1\) and combining with Theorem 2.2 yields the following consequence (illustrated numerically alongside \(V_h\) in Section 6.3 and Figure 4(a)).

\begin{corollary}[Dimension order for the dyadic logistic space]\label{cor:dimension} \emph{For fixed} \(\gamma > 0\) \emph{and all} \(J \ge 1\)\emph{, the space} \(V_J\) \emph{in (8) satisfies} \(L(V_J) \asymp_\gamma 2^J \asymp_\gamma \dim V_J\)\emph{, where} \(\dim V_J = 2^J + 1\)\emph{.}

\end{corollary}

\subsection{Sharp one-sided cross-scale decay}

When two real single-exponential atoms share an endpoint center, their product is even about that endpoint, so their half-line inner product equals half the whole-line inner product and decays exponentially in \(n\). To saturate the \(n^{-1/2}\) bound in Lemma 3.1, we displace the finer-scale center from the endpoint by \(O(h/n)\).

\begin{proposition}[Displaced centers attaining the one-sided rate]\label{prop:displaced-centers} \emph{For a single exponential spectrum, fix} \(a > 0\) \emph{and an integer} \(\beta \ge 0\)\emph{. Let} \(0 < r \le \lambda^{-1}\)\emph{,} \(H = 1/2\)\emph{,} \(h = r H\)\emph{, and} \(c_n = a h \pi / (2n)\)\emph{, where} \(n = q + \beta + 1\)\emph{. Then, as} \(q \to \infty\)\emph{,}

\[
\int_{-\infty}^0 \psi_{h, c_n}(x) \psi_{H, 0}(x)\,dx = -\frac{\sqrt{r}}{\sqrt{\pi n}} \left(\frac{1}{1 - r} + \frac{(-1)^q}{1 + r}\right) + O_r(n^{-1}). \tag{29}
\]

\emph{Consequently, the} \(n^{-1/2}\) \emph{term in (12) cannot be improved to} \(o(n^{-1/2})\) \emph{uniformly over admissible configurations.}

\end{proposition}

\begin{proof} Set \(\theta_0 = \pi / 2\), \(u_0 = \theta_0 / n\), and \(z = u - u_0 \ge 0\), so that \(v = r z\). Proceeding as in the proof of Lemma 3.1, we write \(\int_{-\infty}^0 \psi_{h, c_n} \psi_{H, 0}\) as

\[
\frac{\sqrt{r}}{J_n} \operatorname{Re} \int_0^\infty A_n(z) \bigl(e^{in\Phi_-(z)} + (-1)^q e^{in\Phi_+(z)}\bigr)\,dz,
\]

where

\[
A_n(z) = \bigl((1 + (z + \theta_0/n)^2)(1 + r^2 z^2)\bigr)^{-n/2}, \qquad \Phi_\pm(z) = \arctan(z + \theta_0/n) \pm \arctan(rz).
\]

Fix \(\delta > 0\) small enough that \(\Phi'_\pm(z)\) stays bounded away from zero on \([0, \delta]\) for all large \(n\), and choose a smooth cutoff \(\eta \in C_c^\infty([0, \delta))\) with \(\eta \equiv 1\) on \([0, \delta/2]\). The integral of \(A_n(1 - \eta)\) over \([\delta/2, \infty)\) is exponentially small in \(n\), while integrating the cutoff integral by parts twice (using \(\eta(0) = 1\) and \(\eta'(0) = 0\)) gives

\[
\operatorname{Re} \int_0^\infty A_n(z) e^{in\Phi_\pm(z)}\,dz = -\frac{A_n(0) \sin(n\Phi_\pm(0))}{n \Phi'_\pm(0)} + O_r(n^{-3/2}). \tag{30}
\]

Indeed, on \([0, \delta]\) there exists \(c_* > 0\) such that

\[
A_n(z) \le C e^{-c_* n z^2}, \qquad \vert{}A'_n(z)\vert{} \le C(n z + 1) e^{-c_* n z^2}, \qquad \vert{}A''_n(z)\vert{} \le C\bigl(n + (n z + 1)^2\bigr) e^{-c_* n z^2},
\]

whence \(A'_n(0) = O(1)\), \(\int_0^\delta \vert{}A'_n\vert{} = O(1)\), and \(\int_0^\delta \vert{}A''_n\vert{} = O(n^{1/2})\); since the first three derivatives of \(\Phi_\pm\) are uniformly bounded and derivatives of \(\eta\) are supported in \([\delta/2, \delta]\), the second-order integration-by-parts remainder is \(O_r(n^{-2} n^{1/2}) = O_r(n^{-3/2})\).

Substituting \(A_n(0) = 1 + O(n^{-1})\), \(n \Phi_\pm(0) = \pi/2 + O(n^{-2})\), and \(\Phi'_\pm(0) = 1 \pm r + O(n^{-2})\) into (30) and dividing by \(J_n\) via (10) yields (29). The factor in parentheses in (29) equals \(2/(1 - r^2)\) for even \(q\) and \(2r/(1 - r^2)\) for odd \(q\), both strictly positive, while \(w(c_n/(ah)) w(0) \to 1\) and \(\{c_n\}, \{0\} \subset [0,1]\) satisfy (2) for all large \(n\). \end{proof}

\subsection{Sharp asymptotics of the interval coefficient lower bound}

Combining the exterior interaction in Proposition 4.3 with the interval energy of the displaced atom \(\psi_{h, c_n}\) yields the exact asymptotic rate in Theorem 2.3 for a single exponential spectrum.

\begin{corollary}[Two-atom Gram asymptotics and optimal uniform rate]\label{cor:gram-sharp} \emph{Under the assumptions of Proposition 4.3, let} \(\mathcal{G}_q\) \emph{be the interval Gram matrix of the ordered pair} \((\psi_{h, c_n}, \psi_{H, 0})\)\emph{, and set}

\[
\alpha = \frac{\sqrt{\pi}}{2}, \qquad \chi_q(r) = \sqrt{\frac{r}{\pi}} \left(\frac{1}{1 - r} + \frac{(-1)^q}{1 + r}\right).
\]

\emph{Then, as} \(q \to \infty\)\emph{,}

\[
\mathcal{G}_q = \frac{1}{2} I_2 + \frac{1}{\sqrt{n}} \begin{pmatrix} \alpha & \chi_q(r) \\ \chi_q(r) & 0 \end{pmatrix} + O_{a,r}(n^{-1}) \tag{31}
\]

\emph{in matrix operator norm. In particular,}

\[
\lambda_{\min}(\mathcal{G}_q) = \frac{1}{2} - \frac{\kappa_q(r)}{\sqrt{n}} + O_{a,r}(n^{-1}), \qquad \kappa_q(r) = \frac{\sqrt{\alpha^2 + 4 \chi_q(r)^2} - \alpha}{2} > 0. \tag{32}
\]

\emph{Consequently, for the fixed single-exponential kernel and fixed} \((\lambda, \Delta)\)\emph{, the worst-case Riesz lower constant}

\[
\underline{\Lambda}_q = \inf_{\mathcal{C}} \lambda_{\min}\bigl(G_I^{(q)}(\mathcal{C})\bigr), \tag{33}
\]

\emph{taken over all nonempty finite configurations} \(\mathcal{C}\) \emph{satisfying (2), obeys}

\[
\frac{1}{2} - \underline{\Lambda}_q \asymp_{a, \beta, \lambda, \Delta} n^{-1/2}. \tag{34}
\]

\end{corollary}

\begin{proof} With \(u_n = \pi / (2n)\), squaring (9) gives the pointwise density

\[
a h\,\psi_{h,c}(c + a h u)^2 = \frac{(1 + u^2)^{-n}}{J_n} \bigl(1 + (-1)^q \cos(2n \arctan u)\bigr).
\]

Since \(H = 1/2\), \(h = rH\), and \(c_n \to 0\), the right-exterior tails beyond \(x = 1\) decay as \(O_{a,r}(e^{-\vartheta n})\) for some \(\vartheta > 0\) by the estimate in (19). Evenness of each squared atom about its center therefore gives \((\mathcal{G}_q)_{22} = 1/2 + O_{a,r}(e^{-\vartheta n})\) and

\[
(\mathcal{G}_q)_{11} = \frac{1}{2} + \frac{1}{J_n} \int_0^{u_n} (1 + u^2)^{-n} \bigl(1 + (-1)^q \cos(2n \arctan u)\bigr)\,du + O_{a,r}(e^{-\vartheta n}).
\]

Substituting \(\sigma = n u \in [0, \pi/2]\) in the short integral, we have \((1 + \sigma^2/n^2)^{-n} = 1 + O(n^{-1})\) and \(2n \arctan(\sigma/n) = 2\sigma + O(n^{-2})\) uniformly on \([0, \pi/2]\). Since \(\int_0^{\pi/2} \cos(2\sigma)\,d\sigma = 0\), invoking (10) gives

\[
(\mathcal{G}_q)_{11} - \frac{1}{2} = \frac{1}{n J_n} \left(\frac{\pi}{2} + O(n^{-1})\right) = \frac{\sqrt{\pi}}{2\sqrt{n}} + O_{a,r}(n^{-3/2})
\]

for both parities of \(q\). Meanwhile, the whole-line inner product is bounded in modulus by \((2\sqrt{r}/(1 + r))^{2n-1}\), and the right-exterior inner product on \((1, \infty)\) is \(O_{a,r}(e^{-\vartheta n})\) by Cauchy–Schwarz. Subtracting the exterior integrals from the whole-line inner product via Proposition 4.3 gives \((\mathcal{G}_q)_{12} = \chi_q(r) n^{-1/2} + O_{a,r}(n^{-1})\), proving (31).

Because the smaller eigenvalue of the \(2 \times 2\) leading perturbation matrix in (31) is \(-\kappa_q(r)\), Weyl's inequality yields (32), where \(\kappa_q(r) > 0\) for both parities since

\[
\chi_q(r) = \frac{2\sqrt{r/\pi}}{1 - r^2} \begin{cases} 1, & q \text{ even}, \\ r, & q \text{ odd}. \end{cases}
\]

Finally, Proposition 3.2 (20) gives \(\underline{\Lambda}_q \ge 1/2 - C_{a, \lambda, \Delta} n^{-1/2}\), while choosing \(r = \lambda^{-1}\) in (32) yields an admissible two-atom configuration for all large \(q\) with \(\underline{\Lambda}_q \le 1/2 - c_* n^{-1/2}\) for some \(c_* > 0\), establishing (34). \end{proof}

Both asymptotic expansions (29) and (32) are verified numerically in Section 6.1 (Figure 2). Note that when a positive spectral mixture possesses a spectral gap above its minimal decay parameter (\(\inf_{A_j > 1} A_j > 1\), as in the logistic derivative kernel where \(A_1 = 3\)), the perturbation tail \(\eta_q\) in Section 3.3 decays exponentially in \(q\) and the \(O(n^{-1/2})\) uniform deficit rate persists; when \(\{A_j\}_{A_j > 1}\) accumulates at \(1\), \(\eta_q\) may govern the convergence rate instead.

\section{Deterministic least-squares sampling stability}

Under (2), within-scale separation and geometric scale decay yield the dimension upper bound

\[
D := \dim V \le 2 + (1 + \Delta^{-1}) \sum_{\ell=1}^L h_\ell^{-1} \le 2 + \frac{\lambda(1 + \Delta^{-1})}{\lambda - 1} h_*^{-1}. \tag{35}
\]

Because (35) bounds \(D\) only from above, \(h_*^{-1}\) cannot in general be controlled by \(D\): while a finest-scale density condition \(\# X_L \ge \kappa h_*^{-1}\) implies \(h_*^{-1} \le D / \kappa\) and hence \(L(V) \lesssim D\), Theorem 4.1 rules out any dimension-only inverse bound in the sparse regime.

Let \(\nu = \sum_{i=1}^N \omega_i \delta_{x_i}\) be a discrete probability measure supported on \([0,1]\) (\(x_i \in [0,1]\), \(\omega_i \ge 0\), \(\sum_{i=1}^N \omega_i = 1\)), with star discrepancy and discrete \(L^2(\nu)\) seminorm

\[
d(\nu) = \sup_{0 \le t \le 1} \vert{}\nu([0, t]) - t\vert{}, \qquad \left\lVert g \right\rVert_\nu^2 = \sum_{i=1}^N \omega_i \vert{}g(x_i)\vert{}^2.
\]

For any \(L^2(I)\)-orthonormal basis \(\{\phi_1, \dots, \phi_D\}\) of \(V\), write \(G_\nu = \bigl(\langle \phi_j, \phi_k \rangle_\nu\bigr)_{1 \le j, k \le D}\) for the associated empirical sampling Gram matrix.

\begin{proposition}[Comparison of discrete and continuous norms]\label{prop:sampling} \emph{Every} \(g \in V\) \emph{satisfies}

\[
\left\lvert \left\lVert g \right\rVert_\nu^2 - \left\lVert g \right\rVert_2^2 \right\rvert \le \varepsilon \left\lVert g \right\rVert_2^2, \qquad \varepsilon = 2 d(\nu) C h_*^{-1}, \tag{36}
\]

\emph{where} \(C\) \emph{is the inverse constant in Theorem 2.2. In particular, if} \(\varepsilon < 1\)\emph{, point evaluation on} \(\{x_i\}_{i=1}^N\) \emph{is injective on} \(V\)\emph{, and the spectrum of} \(G_\nu\) \emph{is contained in} \([1 - \varepsilon, 1 + \varepsilon]\)\emph{.}

\end{proposition}

\begin{proof} For any absolutely continuous function \(G\) on \([0,1]\), integrating \(G(x) - G(0) = \int_0^1 \mathbf{1}_{\{t < x\}} G'(t)\,dt\) with respect to \(\nu\) and Lebesgue measure (noting \(\nu([0,1]) = 1\) and \(\nu((t, 1]) = 1 - \nu([0, t])\)) and applying Fubini's theorem yields

\[
\int_{[0,1]} G\,d\nu - \int_0^1 G(t)\,dt = -\int_0^1 \bigl(\nu([0, t]) - t\bigr) G'(t)\,dt,
\]

which remains valid in the presence of point masses at \(\{0, 1\}\). Setting \(G = g^2\) and applying Cauchy–Schwarz together with Theorem 2.2 gives

\[
\left\lvert \left\lVert g \right\rVert_\nu^2 - \left\lVert g \right\rVert_2^2 \right\rvert \le 2 d(\nu) \left\lVert g \right\rVert_2 \left\lVert g' \right\rVert_2 \le \varepsilon \left\lVert g \right\rVert_2^2,
\]

from which the spectral inclusion for \(G_\nu\) follows by expanding \(g\) in \(\{\phi_j\}_{j=1}^D\). \end{proof}

Proposition 5.1 follows the classical passage from Markov–Bernstein inequalities to deterministic least-squares stability \cite[Section~3]{ref2}. Equal-weight interior nodes \(x_i = i / (N + 1)\) achieve \(d(\nu) = 1 / (N + 1)\), uniform midpoint nodes \(x_i = (i - 1/2)/N\) (denoting the corresponding midpoint measure by \(\nu_N = \frac{1}{N}\sum_{i=1}^N \delta_{(i-1/2)/N}\)) achieve \(d(\nu_N) = 1 / (2N)\), and jittered sampling (one point per uniform subinterval \([(i-1)/N, i/N]\)) satisfies \(d(\nu) \le 1 / N\). In particular, midpoint sampling guarantees \(\operatorname{spec}(G_{\nu_N}) \subset [1 - \varepsilon_0, 1 + \varepsilon_0]\) for a prescribed tolerance \(0 < \varepsilon_0 < 1\) whenever \(N \ge C \varepsilon_0^{-1} h_*^{-1}\) (and interior nodes whenever \(N + 1 \ge 2 C \varepsilon_0^{-1} h_*^{-1}\)).

For the dyadic logistic space \(V_J\) of Example 2.7 (8), \(h_*^{-1} = D - 1\), so \(N \gtrsim \varepsilon_0^{-1} D\) samples suffice—matching the algebraic injectivity threshold \(N \ge D\) up to constants depending only on \(\gamma\) and \(\varepsilon_0\). For general sparse configurations, \(N \gtrsim \varepsilon_0^{-1} h_*^{-1}\) governs stable recovery in \(L^2(I)\) (see Section 6.3 and Figure 4(b)), whereas the numerical condition number of the raw collocation matrix in the primitive basis (3) depends additionally on basis conditioning.

\section{Numerical experiments}

We now present three sets of numerical experiments that illustrate and quantitatively validate the main theoretical results: the two-atom boundary-layer asymptotics and \(O(n^{-1/2})\) eigenvalue deficit of Section 4 (Figure 2), the multiscale high-derivative Riesz bounds of Theorem 2.3 (Figure 3), and the trial-space inverse constants and deterministic midpoint sampling stability of Theorem 2.2, Theorem 4.1, and Proposition 5.1 (Figure 4).

\subsection{Two-atom boundary-layer asymptotics and eigenvalue deficit}

To test Proposition 4.3 and Corollary 4.4, consider the single-exponential Cauchy kernel \(K(t) = [\pi(1 + t^2)]^{-1}\) (\(a = 1\), \(\beta = 0\), so \(n = q + 1\)) with coarse scale \(H = 1/2\), scale ratio \(r = 0.5\) (\(h = r H = 1/4\)), and displaced finer-scale center \(c_n = a h \pi / (2n)\). Write \(\langle f, g \rangle_{(-\infty, 0)} = \int_{-\infty}^0 f(x) g(x)\,dx\) for the left-exterior inner product and \(\mathcal{G}_q\) for the \(2 \times 2\) interval Gram matrix of \((\psi_{h, c_n}, \psi_{H, 0})\) on \(I = (0,1)\). We evaluate even derivative orders \(q \in \{32, 64, 128, 256, 512, 1024\}\) (\(n \in \{33, \dots, 1025\}\)) and adjacent odd orders \(q \in \{31, 63, 127, 255, 511, 1023\}\) (\(n \in \{32, \dots, 1024\}\)).

\begin{figure}[tbp]
\centering
\includegraphics[width=\linewidth]{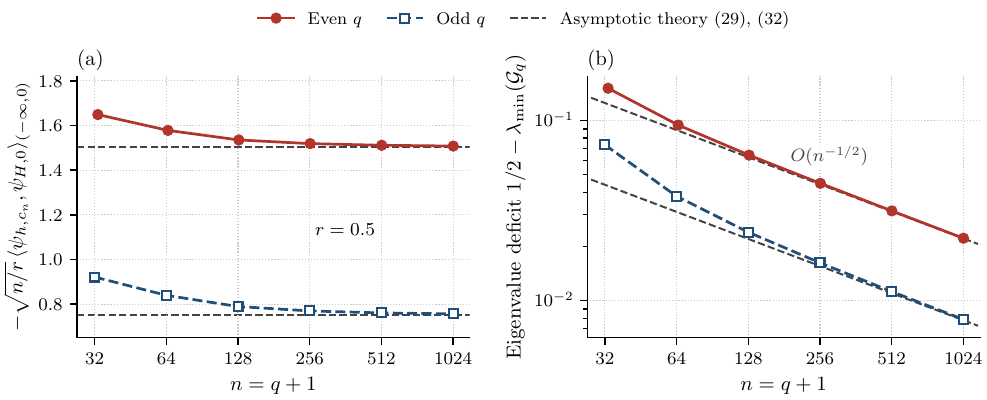}
\caption{\emph{Sharp boundary-layer asymptotics and eigenvalue deficit for displaced two-atom configurations.} Single-exponential Cauchy atoms (\(a = 1\), \(\beta = 0\), \(n = q + 1\)) with \(H = 1/2\), \(r = 0.5\) (\(h = r H\)), and \(c_n = a h \pi / (2n)\) along even-\(q\) (solid, circles) and odd-\(q\) (dashed, squares) subsequences. \textbf{(a)} Scaled exterior half-line inner product \(-\sqrt{n/r}\,\langle \psi_{h, c_n}, \psi_{H, 0} \rangle_{(-\infty, 0)}\) converging to the positive parity limits in (29) (horizontal dashed lines). \textbf{(b)} Log–log decay of the \(2 \times 2\) interval Gram eigenvalue deficit \(1/2 - \lambda_{\min}(\mathcal{G}_q)\) matching the exact \(O(n^{-1/2})\) asymptotic lines \(\kappa_q(r) n^{-1/2}\) from (32) (dashed lines).}
\label{fig:figure2}
\end{figure}

Figure 2(a) confirms that \(-\sqrt{n/r}\,\langle \psi_{h, c_n}, \psi_{H, 0} \rangle_{(-\infty, 0)}\) converges monotonically along the even- and odd-\(q\) subsequences to the distinct positive limits \(\pi^{-1/2}(1/(1-r) \pm 1/(1+r))\) predicted by (29) (\(1.5045\) for even \(q\) and \(0.7523\) for odd \(q\) at \(r = 0.5\)). Figure 2(b) shows that this \(O(n^{-1/2})\) exterior cross-scale coupling pushes \(\lambda_{\min}(\mathcal{G}_q)\) strictly below \(1/2\), with the eigenvalue deficit \(1/2 - \lambda_{\min}(\mathcal{G}_q)\) following the parameter-free theoretical lines \(\kappa_q(r) n^{-1/2}\) in (32) (\(\kappa_{\mathrm{even}}(0.5) \approx 0.7093\) and \(\kappa_{\mathrm{odd}}(0.5) \approx 0.2492\)) at exact slope \(-1/2\).

\subsection{Multiscale high-order interval Gram matrices}

To examine Theorem 2.3 across coupled scales, we compute the smallest eigenvalue of the full normalized interval Gram matrix

\[
G_I^{(q)} = \left(\int_0^1 \psi_{h_\ell, c}^{(q)}(x) \psi_{h_k, d}^{(q)}(x)\,dx\right)_{(\ell, c), (k, d) \in \mathcal{I}}
\]

across five dyadic scales \(h_\ell = 2^{-\ell - 1}\) (\(1 \le \ell \le 5\)) for both the Cauchy kernel \(K(t) = [\pi(1 + t^2)]^{-1}\) (\(a = 1\), \(\beta = 0\)) and the logistic derivative kernel \(K(t) = \pi e^{-\pi t}/(1 + e^{-\pi t})^2\) (\(\gamma = \pi\), \(a = 1\), \(\beta = 1\)). We compare three admissible center geometries satisfying \(\lambda = \Delta = 2\) at adjacent even and odd orders \(q \in \{8, 9, 16, 17, 32, 33, 64, 65, 128, 129\}\):

\begin{itemize}
\item \textbf{Shared endpoints (\(67\) atoms):} \(X_\ell = \{2 j h_\ell : 0 \le j \le (2 h_\ell)^{-1}\}\), placing centers at \(\{0, 1\}\) on every scale alongside interior cross-scale coincidences.
\item \textbf{Moving endpoints (\(44\) atoms):} with \(M_\ell = \lfloor (3 h_\ell)^{-1} \rfloor\) and \(s_{\ell, q} = \pi h_\ell / [2(q + 1)]\),
\[
X_\ell^{(q)} = \left\{s_{\ell, q} + (1 - 2 s_{\ell, q}) \frac{j}{M_\ell} : 0 \le j \le M_\ell\right\},
\]
so extreme centers approach \(\{0, 1\}\) at distance \(O(h_\ell / q)\) while maintaining spacing at least \(3 h_\ell (1 - \pi/36) > 2 h_\ell\) for \(q \ge 8\).
\item \textbf{Interior centers (\(62\) atoms):} \(X_\ell = \{(2j + 1)h_\ell : 0 \le j < (2 h_\ell)^{-1}\}\), buffering all centers away from \(\{0, 1\}\) by \(h_\ell\).
\end{itemize}

Notice that the \(q\)-dependent center placement in the moving-endpoint family is specified via \(q + 1\) identically for both kernels; unlike the two-atom construction of Corollary 4.4, no center is held fixed at \(\{0, 1\}\).

\begin{figure}[tbp]
\centering
\includegraphics[width=\linewidth]{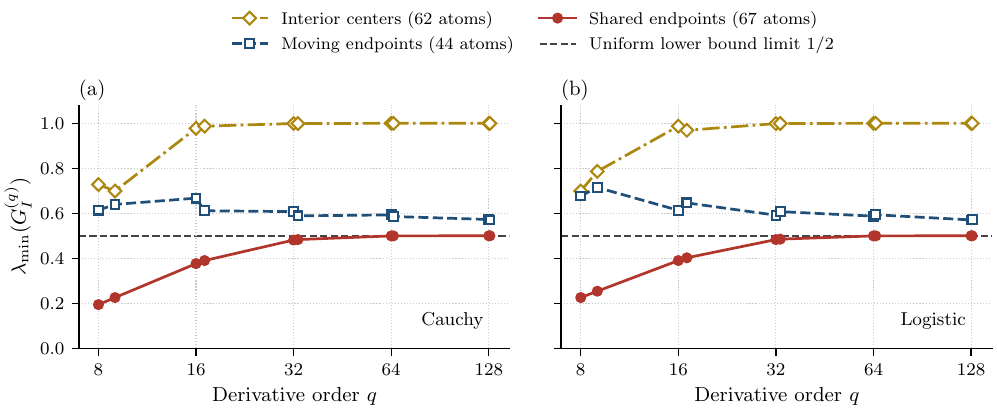}
\caption{\emph{Smallest eigenvalues of multiscale interval Gram matrices across derivative orders.} Minimal eigenvalue \(\lambda_{\min}(G_I^{(q)})\) versus derivative order \(q\) for five geometrically separated scales (\(h_\ell = 2^{-\ell-1}\), \(1 \le \ell \le 5\); \(\lambda = \Delta = 2\), \(a = 1\)) using \textbf{(a)} the Cauchy kernel and \textbf{(b)} the logistic derivative kernel (\(\gamma = \pi\)). Three admissible center geometries are compared: interior centers (\(62\) atoms, dash-dotted, diamonds), moving endpoints (\(44\) atoms, dashed, squares), and shared endpoints (\(67\) atoms, solid, circles). The horizontal dashed line indicates the optimal configuration-uniform lower bound limit \(1/2\) of Theorem 2.3 and Corollary 4.4.}
\label{fig:figure3}
\end{figure}

As shown in Figure 3, all three configurations are already strictly positive definite at \(q = 8\), well before conservative row-sum estimates enter their asymptotic regime. For the shared-endpoint family, \(\lambda_{\min}(G_I^{(q)})\) increases monotonically from \(0.194\) to \(0.500\) for the Cauchy kernel and from \(0.225\) to \(0.500\) for the logistic derivative kernel. By contrast, \(\lambda_{\min}(G_I^{(q)})\) converges rapidly to \(1\) for fixed interior centers (whose diagonal entries tend to \(1\)), whereas the moving-endpoint family exhibits small parity steps between adjacent orders \((q, q+1)\)—reflecting the \((-1)^q\) phase alternation analyzed in Section 4—and approaches \(1/2\) slowly from above at rate \(O(q^{-1/2})\). The close quantitative agreement between the Cauchy and logistic trajectories directly reflects the exponential decay of the mixture perturbation tail \(\eta_q\) in Section 3.3.

\subsection{Trial-space inverse constants and midpoint sampling stability}

Finally, we return from high-order derivatives to the primitive trial space \(V\) itself to examine the sharp inverse constant \(L(V)\) in Theorem 2.2 and the midpoint sampling Gram matrix \(G_{\nu_N}\) in Proposition 5.1. Fixing the logistic primitive \(\sigma(\gamma t)\) with \(\gamma = \pi\) (\(a = 1\)) and separation parameters \(\lambda = \Delta = 2\), we compare three space families across dyadic finest scales \(h_* = 2^{-J}\):

\begin{enumerate}
\item \textbf{Shared-endpoint multiscale space:} scales \(h_\ell = 2^{-\ell}\) (\(1 \le \ell \le J\)) with centers \(X_\ell = \{2j h_\ell : 0 \le j \le 2^{\ell-1}\}\), including \(\{0, 1\}\) on every scale (\(\dim V = 2^J + J + 1\), so \(\dim V = 38\) at \(J = 5\) and \(\dim V = 71\) at \(J = 6\));
\item \textbf{Dyadic multiscale space \(V_J\):} the interior hierarchy of Example 2.7 (8) (\(\dim V_J = 2^J + 1\), so \(\dim V_5 = 33\) at \(J = 5\));
\item \textbf{Two-center space \(V_h\):} the single-scale space of Theorem 4.1 with \(h = h_*\) and interior centers \(1/2 \pm h\) (\(\dim V_h = 4\) for all \(h_*\)).
\end{enumerate}

\begin{figure}[tbp]
\centering
\includegraphics[width=\linewidth]{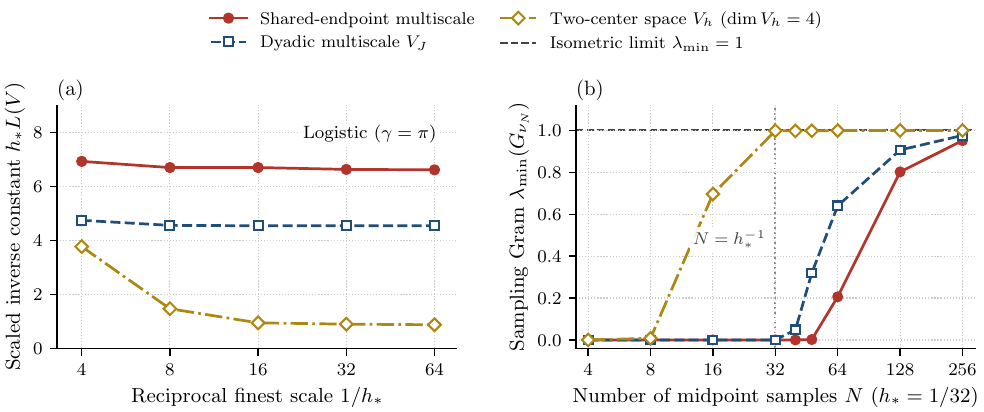}
\caption{\emph{Scaled trial-space inverse constants and deterministic midpoint sampling stability.} Results for logistic primitive spaces (\(\gamma = \pi\), \(\lambda = \Delta = 2\)): shared-endpoint multiscale spaces (solid, circles), dyadic multiscale spaces \(V_J\) (dashed, squares), and four-dimensional two-center spaces \(V_h\) (dash-dotted, diamonds). \textbf{(a)} Scaled inverse constant \(h_* L(V)\) across reciprocal finest scales \(1/h_* \in \{4, 8, 16, 32, 64\}\) (\(2 \le J \le 6\)). \textbf{(b)} Smallest eigenvalue \(\lambda_{\min}(G_{\nu_N})\) of the \(L^2(I)\)-orthonormal empirical sampling Gram matrix under \(N\)-point uniform midpoint sampling at finest scale \(h_* = 1/32\) (\(J = 5\), where \(\lambda_{\min}(G_{\nu_N}) = 0\) whenever \(N < \dim V\)), alongside the scale threshold \(N = h_*^{-1}\) (vertical dotted line) and isometric limit \(\lambda_{\min} = 1\) (horizontal dashed line).}
\label{fig:figure4}
\end{figure}

Figure 4(a) plots the dimensionless quantity \(h_* L(V)\) for \(1/h_* \in \{4, 8, 16, 32, 64\}\). All three families settle onto flat horizontal plateaus—converging to approximately \(6.61\) for the shared-endpoint multiscale space, \(4.54\) for the dyadic space \(V_J\), and \(0.89\) for the four-dimensional space \(V_h\). This simultaneously confirms the uniform \(O(h_*^{-1})\) upper bound of Theorem 2.2 (including stability under multi-scale endpoint coincidence and growing dimension) and the matching \(h_*^{-1}\) lower bound of Theorem 4.1 at fixed dimension \(\dim V_h = 4\).

Figure 4(b) fixes \(h_* = 1/32\) (\(J = 5\)) and displays \(\lambda_{\min}(G_{\nu_N})\) across midpoint sample budgets \(N \in \{4, 8, 16, 32, 40, 48, 64, 128, 256\}\). For the four-dimensional space \(V_h\), even when \(N = 4\) and \(N = 8\) satisfy \(N \ge \dim V_h\), \(\lambda_{\min}(G_{\nu_N})\) remains near zero (\(\le 0.012\)) because grid spacing \(1/N > 2h_*\) undersamples the localized pulse at \(x = 1/2\); stability sets in only as \(N\) approaches \(h_*^{-1} = 32\). For the \(33\)- and \(38\)-dimensional multiscale spaces (where \(\lambda_{\min}(G_{\nu_N}) = 0\) for \(N \le 32\) by rank deficiency), \(\lambda_{\min}(G_{\nu_N})\) rises steeply toward the isometric limit \(1\) once \(N\) exceeds \(h_*^{-1}\) by a modest constant factor (\(N \ge 64\)). Moreover, the left-to-right ordering of the three stability transitions in Figure 4(b) directly mirrors the ordering of \(h_* L(V)\) in Figure 4(a), in quantitative agreement with Proposition 5.1 (36).

\subsection{Quadrature, series truncation, and high-precision verification}

The high-derivative experiments in Sections 6.1 and 6.2 use closed-form analytic derivatives in IEEE double precision (Python 3.13, NumPy 2.5.3, SciPy 1.18.1), without numerical differentiation. Single-exponential Cauchy atoms are evaluated in logarithmic form via (9) to avoid intermediate overflow at large \(n\). For the logistic derivative kernel (\(\widehat{K}(\xi) = 2\vert{}\xi\vert{} \sum_{j \ge 0} e^{-(2j+1)\vert{}\xi\vert{}}\)), applying Plancherel's identity with \(n = q + 2\) and grouping \(j + k + 1 = p\) gives

\[
N_q = \frac{4 \Gamma(2n - 1)}{\pi} \sum_{j, k \ge 0} [2(j + k + 1)]^{-(2n - 1)} = \frac{\Gamma(2n - 1)}{\pi\,2^{2n - 3}} \zeta_{\mathrm{R}}(2n - 2),
\]

where \(\zeta_{\mathrm{R}}(s) = \sum_{p \ge 1} p^{-s}\) is the Riemann zeta function. Combining Fourier inversion with \(J_n = \pi \Gamma(2n - 1) / (2^{2n - 2} \Gamma(n)^2)\) yields

\[
\Psi_q(t) = \left(\frac{2}{J_n \zeta_{\mathrm{R}}(2n - 2)}\right)^{1/2} \operatorname{Re}\left\{i^q \sum_{j=0}^\infty (2j + 1 - it)^{-n}\right\}. \tag{37}
\]

Truncating (37) after \(J_{\mathrm{tr}}\) terms leaves a remainder bounded uniformly in \(t\) by \((2J_{\mathrm{tr}} + 1)^{-n} + (2J_{\mathrm{tr}} + 1)^{1 - n}/[2(n - 1)]\); choosing \(J_{\mathrm{tr}}\) so that the pointwise error in every scaled atom is below \(10^{-17}\) keeps the induced perturbation of \(G_I^{(q)}\) below \(1.8 \times 10^{-16}\) in spectral norm.

In Figure 2, boundary-adapted composite \(36\)-node Gauss–Legendre panels resolving the \(O(a h / n)\) oscillation wavelength near \(x = 0\) agree with refined adaptive quadrature (with upper half-line cutoff increased from \(16\) to \(24\)) to within \(2.0 \times 10^{-11}\). In Figure 3, composite Gauss–Legendre quadrature on uniform panels of width \(h_* / 2\) with \(r_{\mathrm{GL}}(q) = \max\{24, \lceil (q + 2)/2 \rceil + 16\}\) nodes per panel agrees with a refined rule of \(\lceil 3 r_{\mathrm{GL}}(q)/2 \rceil + 8\) nodes to within \(1.4 \times 10^{-13}\) in spectral norm. As independent checks, half-line quadrature of two endpoint Cauchy atoms agrees with half their exact whole-line Gram matrix to within \(3.7 \times 10^{-14}\), and (37) agrees with \(90\)-digit symbolic differentiation of \(K(t)\) at \(q \in \{8, 9, 32, 129\}\) (including \(t = 0\) and \(t = 4\)) to within \(2.8 \times 10^{-14}\).

In Section 6.3 (Figure 4), because raw multiscale primitive bases (3) can become severely ill-conditioned in double precision when many scales share centers, both \(L(V)\) and \(G_{\nu_N}\) are computed in \(90\)-digit arbitrary-precision arithmetic (\texttt{mpmath}) using composite \(20\)-node Gauss–Legendre quadrature on panels of width at most \(h_* / 1.5\), followed by column equilibration and exact Cholesky orthonormalization of the \(L^2(I)\) mass matrix. Doubling both the panel density and working precision changes the plotted values of \(h_* L(V)\) and \(\lambda_{\min}(G_{\nu_N})\) by less than \(10^{-14}\).

\section{Conclusion}

For kernels admitting a positive exponential spectral representation, geometric scale separation and within-scale center separation ensure a sharp interval inverse inequality of order \(h_*^{-1}\), with constants uniform up to the interval endpoints and independent of the number of active scales and centers. Our proof retains exterior diagonal energy in the Gram decomposition \(G_I = G_{\mathbb{R}} - E_- - E_+\), controls exterior off-diagonal interactions via summable endpoint-distance weights, and passes to general positive spectral mixtures through a whole-line synthesis-operator perturbation.

Explicit two-center constructions demonstrate sharpness in two respects: local differences of interior primitives attain the \(h_*^{-1}\) scaling at fixed dimension four for every kernel in the class, while an endpoint–boundary-layer pair for a single exponential spectrum exhibits the exact \(O(n^{-1/2})\) deficit of the uniform high-derivative Riesz lower bound from \(1/2\). As an immediate consequence, deterministic least-squares sampling on \(I\) is stable at sampling budgets proportional to \(h_*^{-1}\), reducing to linear scaling in \(\dim V\) whenever \(h_*^{-1} \asymp \dim V\).

Several natural problems remain open beyond the present framework, including positive exponential spectra without an attained minimal decay parameter, configurations with centers located outside \([0,1]\), and the sharp quantitative dependence of the inverse constant \(C(K, \lambda, \Delta)\) as \(\lambda \downarrow 1\) or \(\Delta \downarrow 0\). Addressing these regimes will require tools beyond the one-sided Gram decomposition developed here.

\end{document}